\documentclass{amsart}

\ifx\pdfoutput\undefined
  \usepackage[dvips]{graphicx}
\else
  \usepackage[pdftex]{graphicx}
\fi

\usepackage{subfigure} 

\usepackage{setspace}
\usepackage{fullpage}
\usepackage{amsmath} 
\usepackage{latexsym}
\usepackage{verbatim}
\usepackage{amsthm}
\usepackage{graphicx}
\usepackage{amsmath,amscd}
\usepackage{amsfonts}
\usepackage{amssymb}
\usepackage{mathrsfs}
\usepackage[all]{xy}
\usepackage{ifpdf}
\usepackage{ifthen}
\usepackage{marginnote}
\usepackage{calc}
\usepackage{ednotes}
\usepackage{comment}
\usepackage{xspace}
\usepackage{ifthen}
\usepackage{xkvltxp}
\usepackage[draft,footnote,index,nomargin]{fixme}
\usepackage{todo}

\usepackage[ps2pdf=true,colorlinks]{hyperref}
\usepackage[figure,table]{hypcap} 
\hypersetup{
   bookmarksnumbered,
   pdfstartview={FitH},
   citecolor={black},
   linkcolor={black},
   urlcolor={black},
   pdfpagemode={UseOutlines}
}

\theoremstyle{definition}
\newtheorem{thm}{Theorem}
\newtheorem{theorem}{Theorem}
\newtheorem{corollary}[thm]{Corollary}

\newtheorem{lemma}[thm]{Lemma}
\newtheorem{prop}[thm]{Proposition}
\newtheorem{definition}[thm]{Definition}
\newtheorem{example}[thm]{Example}
\newtheorem{remark}[thm]{Remark}

\newtheorem{claim}[thm]{Claim}

\numberwithin{mainthm}{section}
\DeclareMathOperator{\C}{\mathcal{C}}
\DeclareMathOperator{\M}{\mathcal{M}}

\DeclareMathOperator{\D}{\mathcal{D}}
\DeclareMathOperator{\coll}{Coll}

\DeclareMathOperator{\Sp}{Sp}

\DeclareMathOperator{\Q}{\mathtt{Q}}
\DeclareMathOperator{\R}{\mathtt{R}}

\DeclareMathOperator{\unit}{\mathbf{1}}
\DeclareMathOperator{\End}{\mathit{End}}
\DeclareMathOperator{\Hom}{\mathit{Hom}}

\DeclareMathOperator{\id}{id}

\DeclareMathOperator{\Ho}{Ho}

\newcommand{\cat}{\mathcal}    
\newcommand{\cA}{{\cat A}}      
\newcommand{\cB}{{\cat B}}      
\newcommand{\cC}{{\cat C}}
\newcommand{\cD}{{\cat D}}

\newcommand{\cH}{{\cat H}}
\newcommand{\cI}{{\cat I}}

\newcommand{\cQ}{{\cat Q}}
\newcommand{\cR}{{\cat R}}
\newcommand{\cS}{{\cat S}}

\newcommand{\cP}{{\cat P}}

\newcommand{\Sets}{{\cat Set}}
\newcommand{\cCat}{{\cat Cat}}
\newcommand{\sSets}{s{\cat Set}}

\newcommand{\multi}{{\cat Multi}}

\DeclareMathOperator{\obj}{obj} 

\DeclareMathOperator{\Alg}{Alg}

\usepackage{color}

\numberwithin{equation}{section}

\begin{document}
\title{The Homotopy Theory of Simplicially Enriched Multicategories}
\author[M. Robertson]{Marcy Robertson}

\address{Departament of Mathematics, University of Western Ontario, Canada}
\email{mrober97@uwo.ca}

         
\begin{abstract}In this paper we construct a cofibrantly generated model category structure on the category of all small symmetric multicategories enriched in simplicial sets.\end{abstract} 
\maketitle

Operads are combinatorial objects that encode a variety of algebraic structures in a particular symmetric monoidal category of interest. Many usual categories of algebras (i.e. categories of commutative and associative algebras, associative algebras, Lie algebras, Poisson algebras, etc.) can be considered categories of operad representations. At the same time, operads are ``algebras'' themselves, or rather monoids in the monoidal category of (symmetric) sequences. A multicategory, or colored operad, is simply an operad with ``many objects,'' analogous to the way a category is a monoid with ``many objects.'' The precise definition and some important examples of operads, multicategories and their algebras will be reviewed in Section $1$ of this paper.\\

The purpose of this paper is the construction of a Quillen model category structure on the category of small multicategories enriched in simplicial sets. Our model structure is a blending of the Bergner model structure on the category of small simplicial categories~\cite{Bergner}, and the Berger-Moerdijk model structure on $S$-colored operads~\cite{BM07}.\\

\subparagraph{Acknowledgments:}The model category structure presented here for simplicially enriched symmetric multicategories was independently obtained by the author as part of her thesis~\cite{Me2} work and Ieke Moerdijk~\cite{MoerdijkUnpubllished} as part of a larger project on $(\infty, 1)$-operads~\cite{MC1,MC2,MC3}. The author is greatly indebted to her thesis advisor, Brooke Shipley, and to Ieke Moerdijk for the many helpful discussions, and to the later for showing her his unpublished manuscript of which she has made liberal use in the preparation of this paper.

\section[Multicategories]{Multicategories}The basic idea of a multicategory is very like the idea of a category, it has objects and morphisms, but in a multicategory the source of a morphism can be an arbitrary sequence of objects rather than just a single object.\\

A \emph{multicategory}, $\mathcal{P},$ consists of the following data:
\begin{itemize}
  \item a set of objects $\obj(\cP)$;
  \item for each $n\ge0$ and each sequence of objects $x_{1},...,x_{n},x$ a \emph{set} $\mathcal{P}(x_{1},...,x_{n};x)$ of $n$-ary operations which take $n$ inputs (the sequence $x_{1},...,x_{n}$) to a single output (the object $x$).\end{itemize}

These operations are equipped with structure maps for units and composition. Specifically, if $I=\{*\}$ denotes the one-point set, then for each object $x$ there exists a unit map $\eta_{x}:I\rightarrow\cP(x;x)$ taking $*$ to $1_{x}$, where $1$ denotes the unit of the symmetric monoidal structure on the category $\Sets$. The composition operations are given by maps$$\mathcal{P}(x_{1},...,x_{n};x)\times\mathcal{P}(y_{1}^{1},...,y_{k_{1}}^{1};x_{1})\times\cdots\times\mathcal{P}(y_{1}^{n},...,y_{k_{n}}^{n};x_{n})\longrightarrow\mathcal{P}(y_{1}^{1},...,y_{k_{n}}^{n};x)$$which we denote by $$p,q_{1},...,q_{n}\mapsto p(q_{1},...,q_{n}).$$The structure maps satisfy the associativity and unitary coherence conditions of monoids. A \emph{symmetric multicategory} is a multicategory with the additional property that the $n$-ary operations are equivariant under the permutation of the inputs. Explicitly, for $\sigma\in\Sigma_{n}$ and each sequence of objects $x_{1},...,x_{n},x$ we have a right action of $\Sigma_{n}$, i.e., a morphism$$\sigma^{*}:\mathcal{P}(x_{1},\cdots,x_{n};x)\longrightarrow\mathcal{P}(x_{\sigma(1)},...,x_{\sigma(n)};x).$$The action maps are well behaved, in the sense that all composition operations are invariant under the $\Sigma_n$-actions, and $(\sigma\tau)^{*}=\tau^{*}\sigma^{*}$.\\

In practice one often uses the following, equivalent, definition of the composition operations, given by:$$\xymatrix{\mathcal{P}(c_{1},\cdots,c_{n};c)\times\mathcal{P}(d_{1},\cdots,d_{k};c_{i})\ar[r]^{{\circ_{i}\,\,\,\,\,\,\,\,\,\,\,}} & \mathcal{P}(c_{1},\cdots,c_{i-1},d_{1},\cdots,d_{k},c_{i+1},\cdots,c_{n};c).}$$ \\

All of our definitions will still make sense if we replace $\Sets$ by any co-complete symmetric monoidal category $(\C,\otimes,\unit)$. Multicategories whose operations take values in $\C$ are called \emph{multicategories enriched in $\C$} or \emph{$\C$-multicategories}. In particular, the strong monoidal functor $\Sets\longrightarrow\mathcal{C}$ that sends a set $S$ to the $S$-fold coproduct of copies of the unit of $\C$ takes every multicategory to a $\mathcal{C}$-enriched multicategory.\\

\begin{example}[Enriched Categories]Let $S$ be a set and let $(\C,\otimes,\unit)$ be a symmetric monoidal category. There exists a (non-symmetric) $S$-colored operad $\cCat_{S}$ whose algebras are the $\C$-enriched categories with $S$ as set of objects and where the maps between algebras (i.e. functors between the $\C$-enriched categories with object set $S$) are the functors which act by the identity on objects. One puts$$\cCat_{S}((x_1,x_1'),\dots,(x_n,x_n');(x_0',x_{n+1}))=\unit$$whenever $x_i'=x_{i+1}$ for $i=0,\dots,n$, and zero in all other cases. In particular, for $n=0$ we have $\cCat_S(;(x,x))=\unit$ for each $x \in S$, providing the $\cCat_S$-algebras with the necessary identity arrows.\end{example}

\begin{example}[Operad Homorphisms]Let $\cP$ be an arbitrary operad.  There exists a colored operad $\cP^1$ on a set $\{0,1\}$ of two colors, whose algebras are triples $(A_0,A_1,f)$ where $A_0$ and $A_1$ are $\cP$-algebras, and $f:A_0 \rightarrow A_1$ is a map of $P$-algebras.  Explicitly,$$\cP^1(i_1,\dots,i_n;i)=\begin{cases}\cP(n)&\text{if }\max(i_1,\dots,i_n)\leq i;\\0&\text{otherwise}.\end{cases}$$The structure maps of $\cP^1$ are induced by those of $\cP$ (for $n=0$, we agree that $\max(i_1,\dots,i_n)=-1)$.  Given a $\cP^1$-algebra on two objects $A_0$ and $A_1$, the objects $\cP(0,\dots,0;0)$ and $\cP(1,\dots,1;1)$ give $A_0$, respectively $A_1$, their $\cP$-algebra structure; furthermore, $\unit \rightarrow \cP(1)$ corresponds to a map $\alpha:\unit \rightarrow \cP^1(0;1)$ giving a map of $\cP$-algebras $f:A_0\rightarrow A_1$. This colored operad has been discussed extensively in the context of chain complexes by Markl \cite{Markl} and plays a key role in describing the derived mapping spaces between multicategories in our paper~\cite{Me2}. \end{example} 

A morphism between two $\C$-enriched, symmetric multicategories $F:\cP\longrightarrow\cQ$, or \emph{multifunctor}, consists of a \emph{set map} of objects $F_{0}:\obj(\cP)\longrightarrow\obj(\cQ)$ together with a family of $\Sigma_n$-equivariant $\C$-morphisms$$\{F:\cP(d_{1},...,d_{n};d)\longrightarrow\cQ(F(d_{1}),...,F(d_{n});F(d))\}_{d_{1},...,d_{n},d\in\cP}$$which are compatible with the composition structure maps. We denote the category of all small symmetric multicategories enriched in $\C$ by $\multi(\C)$. When $\C=\Sets$ we will write $\multi$ rather than $\multi(\Sets)$.\\

\subsection{Enrichment of $\multi$ Over $\cCat$}Multicategories are often called colored operads, or just operads(cf.~\cite{BM06, BV, May, CGMV}, etc.), but we use the term multicategory in this paper because we want to emphasize the relationship between multicategory theory with classical category theory. Informally, we can say that inside every multicategory lies a category which makes up the linear part (i.e. the $1$-operations) of that multicategory. We make this explicit by assigning to each multicategory $\cP$ a category $[\cP]_{1}$ with the same object set as $\cP$ and with morphisms given by $[\cP]_{1}(p,p')=\cP(p;p')$ for any two objects $p,p'$ in $\cP$ (i.e. just look at the operations of $\cP$ which have only one input). The functor $[-]_{1}$ takes all higher operations, i.e. $\cP(p_1,...,p_n;p)$, to be trivial. Composition and identity operations are induced by $\cP$.\\

This relationship with category theory is useful in making sense of ideas which do not have obvious meaning in the multicategory setting. As an example, we will want to identify the ``components'' of a multicategory $\cP$ and we will need a way to say that an $n$-ary operation $\phi$ is ``an isomorphism in $\cP$.'' This is where the relationship between categories and multicategories can be useful, we can say that $\phi$ is an \emph{isomorphism} in $\cP$ if $[\phi]_{1}$ is an isomorphism in the category $[\cP]_{1}$.\\ 

\begin{definition}\label{fullyfaithful}Let $\cP$ and $\Q$ be two multicategories. A multifunctor $F:\cP\rightarrow\Q$ is \emph{essentially surjective} if $[F]_{1}$ is essentially surjective as a functor of categories. We say that $F$ is \emph{full} if for any sequence $p_1,...,p_n,p$ the function $F:\cP(p_1,...,p_n;p)\rightarrow\Q(Fp_{1},...,Fp_{n};Fp)$ is surjective. We say that $F$ is \emph{faithful} if for any sequence $p_1,...,p_n,p$ the function $F:\cP(p_1,...,p_n;p)\rightarrow\Q(Fp_{1},...,Fp_{n};Fp)$ is injective. The multifunctor $F$ is called \emph{fully faithful} if it is both full and faithful.\end{definition}

\begin{definition}Let $F:\cP\rightarrow\Q$ be a functor between two symmetric multicategories. We say that $F$ is an \emph{equivalence of multicategories} if, and only if, $F$ is both fully faithful and essentially surjective.\end{definition}

\begin{definition}Given a multicategory $\mathcal{P}$ with $\obj(\cP):=S$ and a set map $F_0:T\longrightarrow S$, we construct a \emph{pullback} multicategory $F^{*}(\mathbb{\mathcal{P}})$ with object set $T$ whose operations are given by\begin{equation*}\label{pullback multicategory}F^{*}(\mathcal{P})(d_{1},\cdots,d_{n};d):=\mathcal{P}(Fd_{1},...,Fd_{n};Fd).\end{equation*}\end{definition}
 
The functor $[-]_{1}$ admits a left adjoint, denoted by $\Xi(-)$, which takes a category $\C$ to a multicategory $\Xi(\C)$ with $\obj(\Xi\C):=\obj(\C)$. The linear operations are just the composition maps of $\C$, i.e. $\Xi\C(c; c'):=\C(c, c'),$ and the higher operations are all trivial, i.e. $\Xi\C(c_1,...,c_n; c) =0.$ Composition and units are induced from $\C$ in the obvious way, and it is an easy exercise to check the necessary axioms. We apply the \emph{symmetrization functor}, $\underline{S},$ (see next paragraph) to make $\Xi\C$ into a symmetric multicategory, but since all non-trivial operations in $\Xi\C$ are unary (they have only one input), the symmetric groups $\Sigma_n$ can have only trivial actions. So really $\Xi\C$ is both a symmetric and a non-symmetric multicategory.\\

The \emph{symmetrization functor} is the left adjoint to the functor which forgets symmetric group actions, and is given as follows:\begin{equation*}\underline{S}(\Xi\C)(x_1,\ldots,x_n;x) =\coprod_{\sigma\in\Sigma_n}\Xi\C(x_{\sigma^{-1}(1)},\ldots, x_{\sigma^{-1}(n)};x),\end{equation*} so a $n$-operation in $\underline{S}(\Xi\C)(x_1,\ldots,x_n;x)$ consists of an ordered pair $(\phi,\sigma)$ where $\sigma\in\Sigma_n$ and$$\phi\in\cP(x_{\sigma^{-1}(1)},\ldots, x_{\sigma^{-1}(n)};x).$$ We always let $\Sigma_n$ act on the right via the natural group action on the symmetric group coordinate and define composition via the equivariance requirements for a multicategory. We leave it as an exercise to check that $\underline{S}(\Xi\C)$ satisfies the requirements for a multicatgory and that $\underline{S}$ is a left adjoint to the functor which forgets symmetric group actions.\\

\begin{remark}Given categories $\C$ and $\D$, that $F:\C\longrightarrow\D$ is an equivalence of categories if, and only if, $\Xi(F):\Xi\C\longrightarrow\Xi\D$ is an equivalence of multicategories. It is also an easy fact to check that if $F:\cP\longrightarrow\Q$ is an equivalence of multicategories then $[F]_{1}:[\cP]_{1}\longrightarrow[\cQ]_1$ is an equivalence of categories. The opposite statement, however, is not always true (See, for example,~\cite{St}).\end{remark}

\section[Multicategories As Monoids]{Multicategories As Monoids}Operads are monoids in the category of symmetric sequences. A \emph{symmetric sequence} in a symmetric monoidal category $\C$ is a sequence of objects in $\C$, $K = \{K_0, K_1,...,K_n,...\}_{n\in\mathbb{N}}$ where each of the $K_n$ in $\C$ is equipped with an action of the symmetric group $\Sigma_n$. A morphism of $\Sigma_*$-objects $f: K\rightarrow L$ is a sequence of morphisms $f(n): K(n)\rightarrow L(n)$ in the category $\C$ which commute with the group actions. The category of symmetric sequences in $\C$ is denoted by $\C^{\Sigma_*}$.\\

There are two monoidal products on $\C^{\Sigma}$. The first is a symmetric monoidal tensor product given by:\begin{equation*}\label{tensorofsequences}(K\otimes L)(n) = \bigoplus_{p+q=n} (\unit[\Sigma_n]\otimes K(p)\otimes L(q))_{\Sigma_p\times\Sigma_q},\end{equation*}where here $\unit[\Sigma_n]$ is our notation for the sum of $\Sigma_n$-copies of the unit object of~$\C$ and $\Sigma_p\times\Sigma_q\subset\Sigma_{p+q}$ is the group embedding which sends permutations $\sigma\in\Sigma_p$ (respectively, $\tau\in\Sigma_q$) to permutations of the subset $\{1,\dots,p\}\subset\{1,\dots,p,p+1,\dots,p+q\}$ (respectively, $\{p+1,\dots,p+q\}\subset\{1,\dots,p,p+1,\dots,p+q\}$). The group $\Sigma_p\times\Sigma_q$ acts on $\unit[\Sigma_n]$ by translations on the right and the co-invariant quotient makes this right $\Sigma_p\times\Sigma_q$-action agree with the left $\Sigma_p\times\Sigma_q$-action on $K(p)\otimes L(q)$. Thus the object $K(p)\otimes L(q)$ is a $\Sigma_p\times\Sigma_q$-object in $\C$ for all $p,q\ge0$. The group $\Sigma_n$ acts also on $\unit[\Sigma_n]$ by translation on the left. This left $\Sigma_n$-action induces a left $\Sigma_n$-action on $(K\otimes L)(n)$ and determines the $\Sigma_*$-object structure of the collection $\{(K\otimes L)(n)\}_{n\in\mathbf{N}}$.\\

A symmetric sequence is called \emph{trivial}, or \emph{constant}, if $K(r) = 0$ for all $r>0$. The sequence, $\unit,$ given by\begin{equation*}\unit(n) = \begin{cases} \unit, & \text{if $n=0$}, \\ 0, & \text{otherwise}, \end{cases}\end{equation*}defines the unit for the tensor product of symmetric sequences. The associativity of the tensor product of $\Sigma_*$-object is inherited from the base category, and the symmetry isomorphism $\tau(C,D): C\otimes D\rightarrow D\otimes C$ uses the symmetric from $\C$, $C(p)\otimes D(q)\simeq D(q)\otimes C(p),$ and translations by block transposition in $\unit[\Sigma_n]$. We leave the details as an exercise.\\

This symmetric monoidal product is closed, i.e. there is an external hom-object, denoted by $\Hom_{\C^{\Sigma_*}}(K,L)$, for all $K,L\in\C^{\Sigma_*}$ defined by a product\begin{equation*}\Hom_{\C^{\Sigma_*}}(K,L) = \prod_{n=0}^{\infty}\Hom_{\C}(K(n),L(n))^{\Sigma_n}.\end{equation*}Here $\Hom_{\C}(C(n),D(n))^{\Sigma_n}$ denotes the invariant sub-object of $\Hom_{\C}(C(n),D(n))$ under the action of the symmetric group. 

The category $\C^{\Sigma_{*}}$ also admits a \emph{non-symmetric} monoidal product, which we will call the \emph{composition product}, or the \emph{circle product.} The composition product $K\circ L$ is a generalized symmetric tensor construction given by considering the co-invariants of the tensor products $K(r)\otimes L^{\otimes r}$ under the action of the symmetric groups $\Sigma_r$: \begin{equation*}\label{circleproduct} K\circ L = \bigoplus_{r=0}^{\infty} (K(r)\otimes L^{\otimes r})_{\Sigma_r}\end{equation*} where we use the internal tensor product of $\C^{\Sigma_*}$ to form the tensor power $L^{\otimes r}$, and the external tensor product to form the object $K(r)\otimes M^{\otimes r}.$ We should mention that we are assuming the existence of colimits in $\C$ to form the co-invariant object $(K(r)\otimes L^{\otimes r})_{\Sigma_r}$.\\

The category of symmetric sequences with the circle product forms a monoidal category $(\C^{\Sigma_*},\circ, I)$. An \emph{operad} in $\C$ is a monoid in $(\C^{\Sigma_*},\circ, I)$. 

\subsection[Collections]{Collections}Informally, a collection is a symmetric sequence with ``many objects.'' There is not a direct algebraic analogy, but a collection is similar in concept to that of a quiver in Lie theory, in the sense that a quiver is a directed graph serving as the domain of a representation (defined as a functor).\\

Let $S$ be a set. A \emph{collection in $\C$ on the set $S$} is a family of $\C$-objects $K(x_1,...,x_n; x)$ for each sequence $x_1,...x_n;x\in S$ and each $n\ge0$. A morphism of collections, $F:(S,K)\rightarrow(S',K')$, consists of a map of sets $f:S\rightarrow S'$ and for each $n\ge 0$ and each sequence $x_1,...x_n;x$ in $S$ a family of $\C$-morphisms $\{f_{n}:K(x_1,...,x_n;x)\rightarrow K'(fx_1,..., fx_n;fx)\}_{n\ge0}$. We denote by $\coll(\C)$ the category of all collections in $\C$. If we fix the object set $S$, we can consider the category of all collections with $S$-objects, which we denote by $\coll(\C)_{S}$.\footnote{Note that in the category of $\C$-collections with fixed sets of objects, a morphism $F:(S,K)\longrightarrow (S,K')$ is the identity map on objects.} A collection $K$ is called \emph{pointed} if it is equipped with unit maps $1_x:\unit\to K(x;x)$, for all $x\in S$. We have a similar notion of a \emph{collection with symmetric action}, where each $K(x_1,...,x_n;x)$ comes equipped with a $\Sigma_{n}$-action. If $K$ is a collection with symmetric action in $\C$, and $S=\{\star\}$, then $K$ is just a symmetric sequence in $\C$. We denote the category of collections (with or without symmetric action) with $S$-objects by $\coll(\C)_{S}$ and the category of pointed collections with $S$ objects by $\coll(\C)^*_{S}$.\\ 

As with symmetric sequences, the (symmetric) monoidal product of the enriching category $(\C, \otimes, \unit)$ induces a pointwise \emph{tensor product} on $\coll(\C)_{S}.$ The category $\coll(\C)_{S}$ with this pointwise tensor product forms a closed symmetric monoidal category over $\C$. There is an additional monoidal product on $\coll(\C)_{S}$, called the \emph{circle product}, also known as the \emph{composition product}, which is associative, but (highly) non-commutative. This monoidal structure defines a \emph{right closed} monoidal structure over $\C$. For explicit descriptions of this product, see Berger-Moerdijk~\cite[Appendix]{BM07}.  For a set of objects $S$, we define the category of \emph{$\C$-enriched multicategories with fixed object set $S$}, or \emph{$S$-colored operads}, as the category of unitary, associative monoids in $\coll(\C)_{S}$ with respect to the composition product.\\

\section{The Homotopy Theory of $\multi$}One of the most basic examples of a model structure is the  standard model category structure on $\cCat$, the category of all small (non-enriched) categories (See~\cite{RezkCat}).

\begin{theorem}The category $\cCat$ admits a cofibrantly generated model category structure where:
\begin{itemize}
  \item the weak equivalences are the categorical equivalences;
  \item the cofibrations are the functors $F:\cC\rightarrow\cD$ which are injective on objects;
  \item the fibrations are the functors $F:\cC\rightarrow\cD$ with the property that for each object $c$ in $\cC$ and each isomorphism $f:Fc\rightarrow d$ in $\cD$ there exists a $c'$ in $\C$ and an isomorphism $g:c\rightarrow c'$ in $\cC$ such that $F(g)=f$.
\end{itemize}\end{theorem} 

One can use the relationship between $\multi$ and $\cCat$ to construct a cofibrantly generated model category structure on the category of all small symmetric multicategories. The proof we present below appears in several places, but we believe the first occurrence is in the thesis of Weiss~\cite{Weiss}. 

\begin{theorem}The category $\multi$ admits a model category structure where: 
\begin{itemize}
  \item the weak equivalences are the equivalences of multicategories;
  \item the cofibrations are those functors of multicategories $F:\cP\rightarrow\cQ$ which are injective on objects;
  \item the fibrations are those functors of multicategories $F:\cP\rightarrow\cQ$ with the property that for each object $x$ in $\cP$ and for each isomorphism $\phi:F(x)\rightarrow q$ in $\cQ$, there exists an isomorphism $\psi: x\rightarrow x'$ for which $F(\psi)=\phi.$\end{itemize}\end{theorem} 

A multifunctor $F:\cP\rightarrow\cQ$ is a fibration (respectively, cofibration) of multicategories if, and only if, the functor $[F]_{1}$ is a fibration (respectively, cofibration) of categories. In addition, a multifunctor $F:\cP\rightarrow\cQ$ is a trivial fibration if, and only if, the function $\obj(F):\obj(\cP)\rightarrow\obj(\cQ)$ is surjective and $F$ is fully faithful (as a multifunctor). One can think of the fibration condition as being a path-lifting condition for the ``paths'' in $\cQ$ which are isomorphisms.

\begin{proof}Like the model structure for the category of all small categories, we will prove the existence of the model structure on $\multi$ by directly verifying the axioms. 

M1: In their paper~\cite{EM}, Elmendorf-Mandell show that $\multi$ is cocomplete.

M2 and M3: It is easy to verify that if two out of three multi-functors $F$, $G$, $FG$ are equivalences of multicategories than the third is as well. It is also any easy verify that the weak equivalences and cofibrations are closed under retracts. We can then use the fact that fibrations are characterized by a lifting property to conclude that fibrations are also closed under retracts.\\

M4: (Lifting) Consider the following square:
$$\begin{CD}
\cP @>F>> \cR\\
@VIVV     @VVPV\\
\cQ @> G >> \cS
\end{CD}$$where $I$ is a cofibration and $P$ is a fibration. We need to prove the existence of a lift $H$ which makes the diagram commute whenever either $I$ or $P$ is also a weak equivalence. Let us first assume that $P$ is a trivial fibration. Because the function $I_{0}$ is injective (since $I$ is a cofibration) and the function $P_{0}$ is surjective (since $P$ is a fibration) we can define a lift $H_{0}$ at the level of objects. In order to extend this map to the higher arities, we now choose an element $\phi$ in the set $\cQ(x_1,...,x_n;x)$ and consider its image $G(\phi)$ in $\cS(Gx_1,...,Gx_n;Gx)$. Since $P$ is a fully faithful multifunctor and $Hx=G$ on the level of objects, we know that the function $P:\cR(Hx_1,...,Hx_n;Hx)\longrightarrow\cS(Gx_1,...,Gx_n;Gx)$ is an isomorphism. We now define $H(\phi)=x^{-1}(G(\phi))$. One can check that this extends $H$ into the desired lift, and that this extension is unique.\\

Let us now assume that $I$ is a trivial cofibration. Since $I$ is an equivalence of multicategories we may construct a functor $I'$ such that $I'\circ I=id_{\cP}$ together with a natural isomorphism $\alpha: I\circ I'\rightarrow id_{\cQ}$. If we restrict ourselves to the image of $I$ then we can choose $\alpha$ in such a way so that for each object $x$ in $\obj(\cP)$ the component at $IP$ is $\alpha_{Ix}=id_{Ix}$.\\

For any object $x$ in $\cQ$ we have a corresponding object $GII'x$ in  in the image of $P$, i.e. $GII'x=PFI'x$, and $\alpha_{x}:GII'x\rightarrow Gx$ an isomorphism. Since the multifunctor $P$ is a fibration, we know that there exists an object $Hx$ and an isomorphism $\beta_{x}:FI'x\rightarrow Hx$ in the image of $H$ such that $PHx=Gx$ and $P(\beta_x)=G(\alpha_{x})$. If we restrict to objects in $\Q$ which are in the image of $I$, then $HIx=Fx$ and $\beta_{Ix}=id_{Fx}.$ It follows that $H$ gives a lift on objects.\\

Now, let $\phi$ be an $n$-operation in $\cQ(x_1,...,x_n;x)$. Define $H$ as\begin{equation*}\cR(Hx_1,...,Hx_n;FI'x)\buildrel{\beta^{-1}_{x_1},...,\beta^{-1}_{x_n}}\over\longrightarrow\cR(FI'x_1,...,FI'x_n;FI'x)\buildrel{\beta_{x}}\over\longrightarrow\cR(FI'x;Hx).\end{equation*} It is now easily checked that $H$ gives the desired lift.\\

MC5: Given a multifunctor $F:\cP\longrightarrow\cQ$ we want to factor $F$ as $I\circ P$, with $I$ a cofibration and $P$ a trivial fibration. We can construct a multicategory $\cQ'$ with objects $\obj(\cP)\times\obj(\cQ)$ and operations$$\cQ'(x_1,...,x_n;x):=\cQ(\delta(x_1),...,\delta(x_n);\delta(x))$$where $\delta(x)=x$ for $x$ in $\cQ$ and $\delta(x)=Fx$ for $x$ in $\cP.$ A multifunctor $I:\cP\longrightarrow\cQ'$ by letting $I$ be the identity on objects, i.e $Ix=x.$ For $\phi$ an element in $\cP(x_1,...,x_n;x)$ let $I(\phi)=F(\phi)$. One can check that this is a well defined multifunctor, and a cofibration of multicategories. Define the multifunctor $P:\cQ'\longrightarrow\cQ$ on objects by taking $Px=\delta(x)$ and by letting $P$ be the identity on arrows, i.e. for $\phi$ in $\cQ'(x_1,...,x_n;x)$ let $P(\phi)=\phi$. This multifunctor is clearly fully faithful and surjective on objects, and thus a trivial fibration. The case where $F$ factors as a fibration followed by a trivial cofibration is similar. \end{proof}

Notice that all multicategories are both fibrant and cofibrant under this model structure. 

\begin{prop}The adjunction $[-]_1:\multi\leftrightarrows\cCat:\Xi$ is a Quillen adjunction.\end{prop}

\begin{proof} It is straightforward to verify the (much stronger) property that both $\Xi$ and $[-]_{1}$ preserve fibrations, cofibrations, and weak equivalences.\end{proof}

We will now show that the $\multi$ is a cofibrantly generated model category, and the explicit description of the generating cofibrations will be important for our description of the generating cofibrations of simplicially enriched multicategories. Let $\emptyset$ be the initial category with no objects, let $I$ be the category with two objects and a single identity map between them, and let $I'$ be the maximal subcategory of $I$ with the same objects but excluding the map between them. Let $\cH=I\coprod_{I'}I$. The following proposition, together with the fact that fibrations are characterized by a right lifting property, implies that the model category structure on $\multi$ is cofibrantly generated. 

\begin{prop}A multifunctor $F:\cP\longrightarrow\Q$ is a trivial fibration if, and only if, $F$ has the right lifting property with respect to the set of multifunctors\begin{equation*}\{\Xi(\emptyset)\hookrightarrow\Xi(I)\}\bigcup\{G_n[I']\longrightarrow G_n[I] | n\ge 0\}\bigcup\{G_n[\cH]\longrightarrow G_n[I] | n \ge0\}\end{equation*}\end{prop}

\begin{proof}\label{proofmodelstructurenonenriched}For each $n\ge1$ consider the multicategory $G_{n}[I]$ that has $n+1$ objects $\{0,1,...,n\}$ and operations generated by a single arrow from $(1,...,n)\mapsto 0$. Notice that a functor $G_{n}\longrightarrow\cP$ is just a choice of an $n$-operation in $\cP$.\\ 

Let $G_n[I']$ be the sub-operad of $G_n[I]$ which contains the same objects of $G_n[I]$ but only takes the identity operations.\footnote{At $n=1$ the operads $G_n[I]$ and $G_n[I']$ are just $I$ and $I'$, respectively.} It follows that a map $F:\cP\longrightarrow\Q$ has the right lifting property with respect to the inclusion $G_n[I']\hookrightarrow G_n[I]$ if, and only if, the function $F:\cP(x_1,...,x_n;x)\longrightarrow\Q(Fx_1,...,Fx_n;Fx)$is surjective for all $x_1,...,x_n;x$ in $\cP$. 

Now, consider the operad $G_n[\cH]$ which has $n+1$ objects $\{0,1,...,n\}$ generated by \emph{two} different arrows from $(1,...,n)\mapsto 0$. There is an obvious map $G_n[\cH]\longrightarrow G_n[I]$ identifying the two generating arrows of $G_n[\cH]$ with the generating arrow of $G_n[I]$.\footnote{In other words, the map induced by applying the free symmetric sequences construction~\ref{freesymmetricsequences} to the fold map $I\coprod_{I'}I\longrightarrow I$.} A multifunctor $F:\cP\longrightarrow\Q$ has the RLP with respect to $G_n[\cH]\longrightarrow G_n[I]$ if, and only if, the map $F:\cP(x_1,...,x_n;x)\longrightarrow \Q(Fx_1,...,Fx_n;Fx)$ is injective for any sequence $x_1,...,x_n;x$ in $\cP$. In other words, a multifunctor $F:\cP\longrightarrow\Q$ has the RLP (right lifting property) with respect to the set \begin{equation*}\{\Xi(\emptyset)\hookrightarrow\Xi(I)\}\bigcup\{G_n[I']\longrightarrow G_n[I] | n\ge 0\}\bigcup\{G_n[\cH]\longrightarrow G_n[I] | n \ge0\}\end{equation*} if, and only if, $F$ is a trivial fibration.\end{proof}

\begin{prop}The category of all small multicategories together is a monoidal model category with respect to the Boardman-Vogt tensor product.\end{prop}

\section{The Homotopy Theory of $\C$-Multicategories}The enrichment of $\multi$ over $\cCat$ extends to the enriched case with only minor modifications. Let $\C$ be the category of simplicial sets with the standard model structure. Given a simplicial category $\cA$, we can form a genuine category $\pi_0(\cA)$ which has the same set of objects as $\cA$ and whose set of morphisms $\pi_{0}(\cA)(x,y):=[\unit, \cA(x,y)]$. This induces a functor $\pi_{0}(-):\cCat(\C)\rightarrow \cCat,$ with values in the category of small categories and, moreover, a functor $\Ho(\cCat(\C))\longrightarrow\Ho(\cCat).$ In other words, any $F:\cC\longrightarrow\cD$ in $\Ho(\cCat(\C))$ induces a morphism $\pi_{0}(\C)\longrightarrow\pi_{0}(\D)$ which is well defined up to a non-unique isomorphism. This lack of uniqueness will not be an issue for the purposes of this paper since we study properties of functors which are invariant up to isomorphism. The \emph{essential image} of a simplicial functor $F:\cA\rightarrow\mathcal{B}$ is the full simplicial subcategory of $\mathcal{B}$ consisting of all objects whose image in the component category $\pi_{0}(\mathcal{B})$ are in the essential image of the functor $\pi_{0}(F)$. As with the non-enriched case, we can consider the linear part of a $\C$-enriched multicategory $\cP$, $[\cP]_{1}$, which is in this case a simplicial category. Applying the functor $\pi_{0}$ to the simplicial category $[\cP]_{1}$ gives us the \emph{underlying category} of the multicategory $\cP$. In order to cut back on notation, we will just denote this category by $[\cP]_{1}$ rather than $\pi_{0}([\cP]_{1})$. 
   
\subsection{The Berger-Moerdijk Model Structure on $S$-colored Operads}As we mentioned in the first section, an \emph{$S$-colored operad}, or \emph{multicategory enriched in $\C$ with object set $S$}, is a monoid in $\coll(\C)_{S}$ with respect to the composition product(see~\cite[Appendix]{BM07}). As long as our enriching category,~$\C$ satisfies a set of technical conditions (described below), the model category structure on $\coll(\C)^{*}_{S}$ can be lifted along a free-forgetful adjunction to a model structure on $\multi(\C)_{S}$. Intuitively speaking, an $n$-operation of the free multicategory with object set $S$ generated by the pointed collection $K\in\coll(\C)^{*}_{S}$ is a tree with inputs labeled by $1,\dots,n$, edges labeled by objects of $K$, and vertices labeled by elements of $K$. We will not include the explicit construction here, but one can find this construction in~\cite{BM06}, or~\cite{BV}.

\begin{theorem}\cite{BM07}\label{forget}The forgetful functor from $\multi(\C)_{S}$ to \emph{pointed} collections has a left adjoint$$F^{\star}:\coll(\C)^{*}_{S}\rightarrow\multi(\C)_{S}.$$\end{theorem}The category of symmetric sequences, the category of $\C$-collections with fixed set of objects $S$, $\coll(\C)_{S}$, and the category of pointed $\C$-collections with fixed set of objects $S$ are all $\C$-model categories. This follows from the standard argument that given a cofibrantly generated monoidal model category $\C$ with cofibrant unit and given any finite group $G$, there is an induced monoidal model category structure on the category of objects with right $G$-action, $\C^G$ where the forgetful functor $\C^G \longrightarrow\C$ preserves and reflects weak equivalences and fibrations~\cite[Hovey]{Hov99}. In particular, a morphism of collections $K\longrightarrow L$ is a weak equivalence (respectively, fibration) if for each $n\ge0$ and each sequence of objects $x_1,\ldots, x_n;x$ in $S$ the morphism$$K(x_1,\ldots, x_n;x)\longrightarrow L(x_1,\ldots, x_n;x)$$is a weak equivalence (respectively, fibration) in $\C$.\\

A \emph{symmetric monoidal fibrant replacement functor} is a fibrant replacement functor which is symmetric monoidal and for every $X$ and $Y$ in $\C$ the following diagram commutes\begin{equation*}\xymatrix{X\otimes Y\ar[r]^{r_{X\otimes Y}}\ar[d]_{r_X\otimes r_Y} & F(X\otimes Y) \\
FX\otimes F(Y),\ar[ur] &}\end{equation*}where $r\colon\id_{\C}\longrightarrow F$ is the natural transformation coming from fibrant replacement.\\

\begin{theorem}~\cite{BM07}Let $\C$ be a cofibrantly generated monoidal model category with cofibrant unit and a symmetric monoidal fibrant replacement functor. Let $S$ be a fixed set of objects (or colors). If $\C$ has a co-algebra interval, then the category of all non-symmetric $\C$-multicategories with $S$-objects (equivalently, the category of non-symmetric $S$-colored operads) admits a cofibrantly generated model category structure where a morphism $F:\cQ\longrightarrow\cP$ is a weak equivalence (respectively, fibration) if and only if, for each $n\ge0$, and each sequence $x_1,\dots,x_n,x$ of objects in $S$, the map $F:\cQ(x_1,\dots,x_n;x)\longrightarrow\cP(Fx_1,\dots,Fx_n;Fx)$ is a weak equivalence (respectively, fibration) of $\C$-objects. If the interval is moreover cocommutative, the same is true for the category of symmetric $\C$-multicategories with $S$-objects (equivalently, the category of symmetric $S$-colored operads).\end{theorem}

\begin{example}The category of simplicial sets is a Cartesian closed, cofibrantly generated, monoidal model category that admits a co-associative, co-commutative interval. As symmetric monoidal fibrant replacement functor, we can choose either the $Ex^\infty$functor or the singular chain complex of the geometric realization functor, since both are product-preserving. Therefore, for a fixed set $S$, the category of $S$-colored operads form a model category.\end{example}

\section{The Bergner Model Structure on Simplicial Categories}The following theorem is due to Bergner~(\cite{Bergner}).\begin{theorem}The category of all small simplicial categories, $\cCat(\mathcal{C}),$ supports a right proper, cofibrantly generated, model category structure. The weak equivalences (respectively, fibrations) are the $\C$-enriched functors$$F:\mathcal{A}\longrightarrow\mathcal{B}$$ such that:
\begin{description}
  \item[W1] for all objects $x,y$ in $\mathcal{A}$, the $\C$-morphism $F_{x,y}:\mathcal{A}(x,y)\longrightarrow\mathcal{B}(Fx,Fy)$ is a weak equivalence (respectively, fibration) in the model structure on $\C$ and
  \item[W2]the induced functor $\pi_0(F):\pi_{0}(\mathcal{A})\longrightarrow \pi_{0}(\mathcal{B})$ is a weak equivalence (respectively, fibration) of categories.\end{description}\end{theorem} 

Bergner also gives an explicit description of the generating (acyclic) cofibrations of this model structure, which is worth describing here. We let $\emptyset$ denote the empty category and $I=\{*\}$ for the category which has one object and one identity arrow (viewed as a simplicial category by applying the strong monoidal functor $\Sets\longrightarrow\sSets$). For any simplicial set $K$ we define a simplicial category $G_{1}[K]$ which has two objects, arbitrarily called $0$ and $1$, and $\Hom(0,1)=K$ as the only non-zero function complex. 

Note that if $F:\cA\rightarrow\cB$ satisfies condition $W1$, then checking that $F$ satisfies condition $W2$ is equivalent to checking that the induced functor$$\pi_0(F):\pi_0(\cA)\longrightarrow\pi_0(\mathcal{B})$$is essentially surjective.

\begin{prop}\cite{Bergner}\label{generatingcofibrationsbergner}A functor of simplicial categories $F\cA\rightarrow\cB$ is an acyclic fibration if, and only if, $F$ has the right lifting property (RLP) with respect to all the maps\begin{itemize}
\item $G_1[K]\hookrightarrow G_1[L]$ where $K\hookrightarrow L$ is a generating cofibration of $\sSets$ and
\item the maps $\emptyset\hookrightarrow I$.\end{itemize}\end{prop}

\begin{prop}\cite{Bergner}\label{generatingtrivialcofibrationsbergner}A functor of simplicial categories $F\cA\rightarrow\cB$ is a fibration if, and only if, $F$ has the right lifting property (RLP) with respect to all the maps\begin{itemize}
\item $G_1[K]\hookrightarrow G_1[L]$ where $K\hookrightarrow L$ is a generating cofibration of $\sSets$ and
\item the maps $I\hookrightarrow\mathcal{H}$ where $\obj(\mathcal{H})$ is a set of representatives for the isomorphism classes of simplicial categories on two objects. Each function complex of $\mathcal{H}$ is weakly contractable and has countably many simplices. Furthermore, we require that $\mathcal{H}$ be a cofibrant object in the Dwyer-Kan model category structure on $\cCat(\sSets)_{\{x,y\}}$~\cite{DK1}.\end{itemize}\end{prop}

\subsection{The Proof of The Main Theorem} The main theorem of this paper is the following. 
\begin{theorem}The category of small $\C$-enriched symmetric multicategories admits a right proper cofibrantly generated model category structure in which a multifunctor $$F:\cP\rightarrow\Q$$ is a weak equivalence if: 

\begin{description}	
	\item[W1] for any $n\ge 0$ and for any signature $x_1,...,x_n;x$ in $\cP$ the map of $\C$-objects $$F:\cP(x_1,...,x_n;x)\longrightarrow\cQ(Fx_1,...,Fx_n;Fx)$$ is a weak equivalence in the model category structure on $\C$. 	
	\item[W2] the induced functor $[F]_{1}$ is a weak equivalence of categories.
\end{description} A simplicial multifunctor $F:\cP\longrightarrow\Q$ is a fibration if:\begin{description}\label{fibrations}
  	\item[F1] for any $n\ge 0$ and for any signature $x_1,...,x_n;x$ in $\cP$ the map of $\C$-objects $$F:\cP(x_1,...,x_n;x)\longrightarrow\cQ(Fx_1,...,Fx_n;Fx)$$ is a fibration in the model category structure on $\C$. 
	\item[F2] the induced functor $[F]_{1}$ is a fibration of categories.
\end{description} The cofibrations (respectively, acyclic cofibrations) are the multifunctors which satisfy the left lifting property (LLP) with respect to the acyclic fibrations (respectively, fibrations).\end{theorem}  

We have the following useful characterization of acyclic fibrations.

\begin{lemma} If a multifunctor $F:\cP\longrightarrow\Q$ fixes objects, and satisfies conditions $[F1]$ and $[W1]$, then $F$ is a acyclic fibration.\end{lemma}We can give an explicit description of the generating cofibrations and generating acyclic cofibrations.

\begin{definition}[Generating Cofibrations]The set $\bar{I}$ of generating cofibrations consists of the following $\C$-multifunctors closed under pushouts, transfinite composition, and retracts: 
\begin{description}
 
	\item[C1] Given a generating cofibration $K\hookrightarrow L$ in the model structure on $\C$ the induced multifunctors $G_{n}[K]\longrightarrow G_{n}[L]$ for each $n\ge 0$.\footnote{i.e. a cofibration of $\multi(\C)_{\{0,1,...,n\}}$}

	\item[C2] the $\C$-functors $\emptyset\longrightarrow\cI$ viewed as $\C$-multifunctors via $\Xi.$\end{description}\end{definition}\bigskip

\begin{definition}[Generating Acyclic Cofibrations]\label{generatingacycliccofibrations}The set $\bar{J}$ of generating acyclic cofibrations consists of the following $\C$-multifunctors closed under pushouts, transfinite composition, and retracts: 
\begin{description} 
	\item[A1] Given a generating acyclic cofibration $K\hookrightarrow L$ of the model structure on $\C$ the induced multifunctors $G_{n}[K]\longrightarrow G_{n}[L]$ for each $n\ge 0$.  
  	\item[A2] The $\C$-functors $\cI\longrightarrow\cH$ viewed as multifunctors via $\Xi.$
\end{description}\end{definition}

The multicategories $G_{n}[K]$ are represented by the ``corolla''  
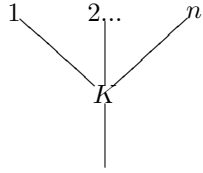
$$\xymatrix{*{1}\ar@{-}[dr]  & *{2...}\ar@{-}[d] & *{n}\ar@{-}[dl]\\
 & *{K}\ar@{-}[d]\\
 &  & *{\,}}$$where $K$ is a simplicial set, and each simplex in $K$ is an operation $1,...,n\mapsto 0$ and there are no other non-identity operations. 

\subsection[Generating Collections]{Generating Collections}The multicategories $G_{n}[-]$ are the free multicategories generated by the generating collections $G_{n}$. The $G_{n}$ form a set of small projective generators for the category $\coll(\C)_{S}$. In order to give a precise definition of these generating collections we must first understand an alternate, but equivalent, description of $\C$-collections. This section is unfortunately abstract, and the reader may want to skip this section and return to it at a later time.\\

Let $S$ be a finite set. We define a $\C$-collection on $S$-objects as a pre-sheaf $K:\mathbb{F}(S)^{op}\rightarrow\C$, where $\mathbb{F}(S)$ is the category whose objects are triples $(X,x_0,\alpha)$, with $X$ a finite set, $x_0\in X$ a chosen base point, and $\alpha:X\rightarrow S$ a chosen function. The morphisms in $\mathbb{F}(S)$ $$(X,x_0,\alpha)\stackrel{f}{\longrightarrow}(X',x'_0,\alpha')$$ are basepoint preserving bijections $f:X\rightarrow X'$ which are compatible with the chosen functions to $S$, i.e. $\alpha'\circ f = \alpha$.\\

Since we know that every finite pointed set is isomorphic to a set of the form $[n+1]=\{1,\dots,n;0\}$ with $0$ viewed as the basepoint, we can always consider the category $\mathbb{F}^\circ(S)$, which is the full subcategory of $\mathbb{F}(S)$ whose objects are finite pointed sets isomorphic to $[n+1]$.  The category $\mathbb{F}^\circ(S)$ is just direct sum of translation groupoids $(S^n\times S)\rtimes \Sigma_n$ where $\Sigma_n$ acts on $S^n\times S$ by permuting the first $n$ coordinates, $(x_1,\dots,x_n;x_{0})^{\sigma}=(x_{\sigma(1)},\dots,x_{\sigma(n)};x_{0})$, and that the inclusion $\mathbb{F}^\circ(S)\longrightarrow\mathbb{F}(S)$ is an equivalence of categories. We can simplify the picture even further by assuming that the object set, $S,$ is equipped with a \emph{linear order}, $\leq$. The category $\mathbb{F}^\leq(S)$ is also equivalent to $\mathbb{F}(S)$, but only has objects $([n+1],0,\alpha)$ for which $\alpha(1)\leq \dots \leq \alpha(n)$ and $\alpha(x_0)=0$. In other words, the category $\mathbb{F}^\leq(S)$ is just $$\coprod_c \coprod_{c_1\leq \dots \leq c_n}\textstyle \sum_{c_1\dots c_n}$$ where $\Sigma_{c_1\dots c_n}\subseteq \Sigma_n$ is the subgroup of permutations $\sigma\in\Sigma_n$ which preserve the order of the sequence, i.e. $c_{\sigma(1)}\leq \dots \leq c_{\sigma(n)}$.\footnote{This just means that $(c_1,\dots,c_n)=(c_{\sigma(1)},\dots,c_{\sigma(n)})$.} Since any finite $S$ can be given a linear order, the category $\mathbb{F}(S)$ is equivalent to one of the form $\mathbb{F}^\leq(S)$.\\

Now, if we consider $\C$-collections to be pre-sheaves $K:(\mathbb{F}^\leq(S))^{op}\longrightarrow\C,$ then for every object $Y$ in $\C$, and for each integer $n\ge 0$ we can define a $\C$-collection where $G_{n}(Y)(x_{1},...,x_{k};x_0)=0$ except in the special case where $n=k$ and $a_{i}=\alpha(i)$ and $x_0=\alpha_0$, in which case we define $G_{n}(Y)(x_{1},...,x_{k};x_0)=Y$. We consider the $G_{n}$ as symmetric multicategories by applying the free multicategory functor $\mathfrak{F}(G_{n})$ (for more on the free multicategory functor see \cite{BM07}, or\cite{EM}).\\

\begin{example}It can be helpful to understand this construction in the one-object case. Let $S=\{\star\}$. Then\begin{equation*}G_n(r) = \begin{cases} \unit[\Sigma_n], & \text{if $r = n$}, \\
0, & \text{otherwise}. \end{cases}\end{equation*} Where $\unit[\Sigma_n]$ denotes the symmetric sequence in~$\C$ formed by the sum over $\Sigma_n$ copies of the tensor unit $\unit\in\C$. The symmetric group $\Sigma_n$ acts on $G_n(n) = \unit[\Sigma_n]$ by translations on the right, and hence acts on $G_n$ on the right by automorphisms of $\Sigma_*$-objects. The symmetric sequences $G_n$, $n\in\mathbb{N}$, are characterized by the following property:\begin{prop}There exists a natural $\Sigma_n$-equivariant isomorphism \begin{equation*}\omega_n(Y): Y(n)\longrightarrow\Hom_{\C^{\Sigma_*}}(G_n,Y),\end{equation*}for all $C\in\C^{\Sigma_*}$.\end{prop}The proof follows immediately from the definitions, namely we have$$\Hom_{\C^{\Sigma_*}}(G_n,Y)\simeq\Hom_{\C}(\unit[\Sigma_n],Y(n))^{\Sigma_n}$$and$$\Hom_{\C}(\unit[\Sigma_n],Y(n))^{\Sigma_n}\simeq\Hom_{\C}(\unit,Y(n))\simeq Y(n).$$ The $\Sigma_n$-action by right translations on $\unit[\Sigma_n]$ corresponds to the internal $\Sigma_n$-action of $Y(n)$ under the isomorphisms in the second line. Hence, we obtain a $\Sigma_n$-equivariant isomorphism $\omega_n(Y):Y(n)\longrightarrow\Hom_{\C^{\Sigma_*}}(G_n,Y)$.\end{example}


We will delay the proof of the model structure momentarily, to prove the following lemmas.

\begin{lemma}[Classifying Fibrations]\label{classifyingfibrations}A $\C$-multifunctor is a fibration if and only if it has the right lifting property with respect to the class of generating acyclic cofibrations.\end{lemma} 
 
\begin{proof} Maps of multicategories are defined ``locally,''i.e. given the multifunctor $F:\cP\longrightarrow\Q$, it is an easy observation that the map of $\C$ objects $$F:\cP(x_1,...,x_n;x)\longrightarrow\Q(Fx_1,...,Fx_n;Fx)$$ is a fibration, if and only if $F$ satisfies the right lifting property (RLP) with respect to multifunctors which are locally acyclic cofibrations, i.e. the set $[A1]$. Furthermore, since the functors $\Xi$ and $[-]_{1}$ form an adjoint pair, we can observe that the multifunctor $F$ satisfies the RLP with respect to the set $[A2]$ if and only if the $\C$-functor $[F]_{1}$satisfies the RLP with respect to the inclusions $\cI\longrightarrow \cH$ which we know to be true by the classification of fibrations in $\cat(\C)$. Putting this together, we conclude that $F$ has the RLP with respect to the set $J$ if and only if it $F$ is a fibration of multicategories. \end{proof} 

\begin{lemma}[Classifying Trivial Fibrations]\label{classifying trivial fibrations}A $\C$-multifunctor is a trivial fibration if and only if it has the right lifting property with respect to the class of generating cofibrations.\end{lemma} 

\begin{proof} The proof is nearly identical to the proof of the previous lemma. \end{proof} 

\begin{lemma} Every acyclic cofibration is a weak equivalence.\end{lemma}
\begin{proof} We break the proof into several smaller claims. 
\begin{claim}[Claim 1:]If $m:S\longrightarrow T$ is an injective map of sets, then the induced functor $\cCat(\C)_{S}\longrightarrow\cCat(\C)_{T}$ preserves cofibrant objects.\end{claim}

\begin{claim}[Claim 2:]\label{claim 2} Consider the following pushout square of simplicial multicategories$$\begin{CD}
\cP @>F>> \cP'\\
@VVJV        @VVGV\\
\cQ @>K>> \cQ'
\end{CD}.$$If $J:\cP\longrightarrow\cQ$ is a trivial cofibration which is bijective on objects, and if $F:\cP\longrightarrow\cP'$ is a multifunctor which is injective on objects. Then $G\cP'\longrightarrow\cQ'$ is also a trivial cofibration which is bijective on objects.\end{claim}
\begin{proof}[Proof of Claim 2:]\label{proofofclaim2}We may assume that the multifunctor $J$ is the identity on objects, and is therefore a member of the class of generating acyclic cofibrations. To ease notation we will denote the object set of $\cP$ by $S$ and the object set of $\cP'$ by $S'=S\coprod T$ where we identify the effect of $F$ on objects with the coproduct of inclusions $S\hookrightarrow S'$. We can adjoin the set $T:=S/F(0)$ of objects to $\cP$, adding no new operations other than identities. This defines a new multicategory, denoted $\cP\coprod(\coprod_{T}\unit)$. In a similar manner, we construct $\cQ\coprod(\coprod_{T}\unit)$, and decompose the pushout$$\begin{CD}
\cP @>F>> \cP'\\
@VVJV        @VVGV\\
\cQ @>K>> \cQ'
\end{CD}$$ as the composition of two pushouts
$$\begin{CD}\cP  @>>>\cP\coprod(\coprod_{T}\unit)  @>>>       \cP'\\
@VVJV          @VVJ\coprod(\coprod_{T}id) V                     @VVGV\\
\cQ @>>> \cQ\coprod(\coprod_{T}\unit)   @>>>  \cQ'.
\end{CD}.$$ Now, by assumption, $J$ is a trivial cofibration of similical multicategories, and thus a trivial cofibration in the model structure on $\multi(\C)_{S}$. It follows that $J\coprod(\coprod_{T}id)$ is a trivial cofibration in the model structure on $\multi(\C)_{S'}$. This implies that $G$ is a weak equivalence in the model structure on $\multi(\C)_{S'}$. Since $G$ is also the pushout of a coifbraion of simplicial multicategories, we have that $G$ is a trivial cofibration of $\multi(\C)$ which is bijective on objects.\end{proof}
\begin{claim}[Claim 3:]\label{claim3}Let $S$ denote a set, and let $s:\{0,...,n\}\twoheadrightarrow S$ be a surjection of sets. Then for any map $X\longrightarrow Y$ of simplicial sets, the square$$\begin{CD}
G_{n}[X] @>>> s_*(G_{n}[X])\\
@VVV        @VVV\\
G_{n}[Y] @>>> s_*(G_{n}[Y])
\end{CD}$$is a pushout of simplicial multicategories. Moreover, if $X\longrightarrow Y$ is a generating trivial cofibration of simplicial sets, then the right vertical map is a weak equivalence in the model category structure on $\multi(\C)_{S}$.\end{claim}
\begin{proof}[Proof of Claim 3]Given $\cP$, a simplicial multicategory with object set $T$, a map $G_n[X]\longrightarrow\cP$ consists of a set map $f:\{0,...,n\}\longrightarrow T$ together with a map $G_n[X]\longrightarrow f^*(\cP)$. Recall that the map $G_{n}[X]\longrightarrow f^{*}(\cP)$ is just a map of simplicial sets $X\longrightarrow\cP(fx_1,...fx_n;fx_0)$. Moreover, a map $s_{*}(G_n[X])\longrightarrow\cP$ is equivalent to a map $f$ which factors as $f=g\circ s$, together with a map $s_*(G_n[X])\longrightarrow g^*(\cP)$\footnote{equivalently, a map $G_n[X]\longrightarrow f^*(\cP)$}. So it is clear that the square is a pushout. If $X\longrightarrow Y$ is a generating acyclic cofibration of simplicial sets, then $G_n[X]\longrightarrow G_n[Y]$ is a generating acyclic cofibration of $\multi(\C)_{S}$. \end{proof}
\begin{remark}Is is essential that our enriching category $\C$ be a monoidal model category since the induced operad $s_*(G_n[X])$ is trivial if $s$ identifies two numbers $i$ and $j$ with $0<i<j$. If, however, $s$ identifies $0$ and $i>0$ we get something more complicated. For example, if $n=1$ and $s$ identifies $0$ and $1$ then $s_*(G_1[X])$ is the free simplicial monoid on $X$.\end{remark} 

\begin{claim}\label{NAFA} Given a multicategory $\cP$ enriched in $\C$, a multifunctor $F:\Xi(\cI)\rightarrow\cP$ and the pushout square: 
$$\begin{CD}
\Xi(\cI) @>F>> \cP\\
@VVJV        @VVGV\\
\Xi(\cH) @>K>> \cQ
\end{CD}$$ then the multifunctor $G$ is fully faithful. \end{claim}

\begin{proof}[Proof of Claim 4] We will proceed by giving an explicit construction of a simplicial multicategory $\cQ$ from the maps $F$ and $J$ and prove that $\cQ$ is the pushout of these maps. It will then be clear from the construction that $G$ is fully faithful. 

As usual, we denote the object set of $\cP$ by $\obj(\cP)$, and the $\C$-enriched category with a single object $*$ by $\cI$. We will fix an object $x_{*}$ of $\cP$ which is in the image of $*$ under $F$. The object set of the multicategory $\cQ$ will be denoted by $\obj(\cQ)=\obj(\cP)\coprod *$.

We want to understand all possible $n$-operations of $\Q$, and to make this easier we start by considering the underlying non-symmetric multicategory, which we also denote by $\Q$. The generating $n$-operations of the non-symmetric multicategory $\Q$ are the following: 

\begin{enumerate}

	\item $\cQ(x_1,...,x_n;x)=\cP(x_1,...,x_n;x).$ 

	\item The $\cQ$ morphisms are given by the $\cP$-morphisms with an action on the left from $\Xi(\cH).$ Explicitly, let $\phi\in\cH(0,1)$, $\psi\in\cP(x_1,...,x_n;x_{*})$, and $\theta\in\cI(*,*)$.  Then $\cQ(x_1,...,x_n;*)$ can be given by $$\cH(0,1)\otimes \cP(x_1,...,x_n;x_{*})/\{\phi\circ\ J(\theta)\otimes\psi=\phi\otimes F(\theta)\circ \psi\}.$$	

	\item There are two instances where we have both a left and a right action coming from $\Xi(\cH)$, which represent $\Q(x_1,...,x_n,\underbrace{*,...,*}_{m};*)$ and $\Q(x_1,...,x_n,\underbrace{*,...,*}_{m};x)$. The $n+m$-operations $\Q(x_1,...,x_n,\underbrace{*,...,*}_{m};*)$ are the $\cP$-morphisms $\cP(x_1,...,x_n, x_*,...,x_*;x_*)$ equipped both left and right actions from $\Xi(\cH)$. The left action from $\phi\in\Xi(\cH)(0,1)$ (as above) and the right actions from $\phi'\in\Xi(\cH)(1,0)$. Explicitly, the operations $$\Xi(\cH)(0,1)\otimes\cP(x_1,...,x_n,\underbrace{x_*,...,x_*}_{m};x_*)\otimes (\underbrace{\Xi(\cH)(1,0),...,\Xi(\cH)(1,0)}_{m})$$ $$\phi\otimes\psi\otimes(\phi^{-1},...,\phi^{-1})\Big/\begin{cases}(\phi\circ J(\theta))\otimes \psi\otimes (\underbrace{\phi^{-1},...,\phi^{-1}}_{m})=\phi\otimes F(\theta)\circ\psi\otimes (\phi^{-1},...,\phi^{-1}),\cr \phi\otimes\psi\otimes(\phi^{-1},...,J(\theta)\circ\phi^{-1},...,\phi^{-1})=\phi\otimes\psi\circ_{i}F(\theta)\otimes(\underbrace{\phi^{-1},...,\phi^{-1}}_{m})\end{cases}\Biggr\}$$ The other case is similar and we leave it as an exercise. 

\item The addition of a unit morphism for the object $*$ in $\obj(\cQ)$. For this we just let $$\phi\in\Xi(\cH)(1,1)=\cQ(*;*).$$

\end{enumerate}

Composition in $\Q$ can be informally described as composition in $\cP$ tensored with composition in $\Xi(\cH)$ modulo relations coming from $\Xi(\cI).$\\

We can check that the multicategory we constructed fits into a commutative diagram 
$$\begin{CD}
\Xi(\cI) @>F>> \cP\\
@VVJV        @VVGV\\
\Xi(\cH) @>K>> \cQ
\end{CD}$$where the map $G$ takes operations in $\cP$ to the $\Q$-operations given by $\cQ(x_1,...,x_n;x)=\cP(x_1,...,x_n;x)$ and the map $K$ can be described as follows:
\begin{equation*}
K(\phi) = \left\{
\begin{array}{rl}
\phi & \text{if } \phi\in\cH(0,0)\\
\phi & \text{if } \phi\in\cH(1,1)\\
\phi\otimes\unit_{x_0} \in\Q(x_0,*)& \text{if } \phi\in\cH(0,1),\\
\unit_{x_0}\otimes\phi \in\Q(*,x_0)& \text{if } \phi\in\cH(1,0),\\
\end{array} \right.\end{equation*} If we are given another multicategory $\R$ and two multifunctors $F':\cP\rightarrow\R$ and $F'':\Xi(\cH)\rightarrow\R$ which satisfy $F''\circ J=F'\circ F$, then we can define a multifunctor $G':\Q\rightarrow\R$ as follows. On objects $G'_0(x_i)=F'_0(x_i)$ and $G'_0(*)=F''_0(1)$. On operations we have 
\begin{equation*}
\begin{array}{lc}
G'(\psi)=F'(\psi) \\
G'(\phi\otimes\psi)=F''(\phi)\circ F'(\psi) \\
G'(\phi\otimes\psi\otimes(\phi_1,...,\phi_m))=F''(\phi)\circ F'(\psi)\circ(F''(\phi_1),...,F''(\phi_m))\\
G'(\psi\otimes(\phi_1,...,\phi_m))= F'(\psi)\circ(F''(\phi_1),...,F''(\phi_m))\\
G'(\phi)=G(\phi)\end{array} \end{equation*} The reader can now check that $G'$ is a well-defined multifunctor and that $G'$ is the \emph{unique} multifunctor which satisfies $G'\circ G=F''$ and $G'\circ F=F'$. This proves that the commutative diagram given above is a pushout, and is now clear that $G$ is fully faithful.\end{proof}

\begin{proof}[Proof of Lemma]\label{proofclassifyingacyliccof} It is enough to show that the pushout of a multifunctor from either the set $A1$ or $A2$ is a weak equivalence, i.e. that if $J:\cA\rightarrow\cB$ is in either the set $A1$ or the set $A2$ and the square
$$\begin{CD}
\cA @>F>> \cP\\
@VVJV        @VVPV\\
\cB @>G>> \cQ
\end{CD}$$ is a pushout square in $\multi(\C)$, then $G$ is a weak equivalence. 

We will split the proof into two cases. Let's first assume that $J$ is an acyclic cofibration from the set $A1$, and consider the following pushout diagram:
$$\begin{CD}
G_{n}[K]   @>F>> \cP\\
@VVJV             @VVPV\\
G_{n}[L]   @>G>>    \Q
\end{CD}.$$ 

Recall that the multifunctor $F:G_{n}[K]\rightarrow\cP$ consists of no other data except a set map $F_{0}:\{0,...,n\}\rightarrow \obj(\cP)$ together with an $\C$-morphism $K\rightarrow\cP(F(1),...,F(n);F(0)).$ The set map $F_0$ can be factored into and injection followed by a surjection, $$\{0,...,n\}\hookrightarrow S'\twoheadrightarrow S$$. Now, since $K\hookrightarrow L$ is a generating acyclic cofibration of $\C$, we  factor the pushout square into two pushouts 
$$\begin{CD}G_n[K]  @>>> s_*(G_n[K]  @>>>       \cP\\
@VVJV          @VVV                     @VVGV\\
G_n[L] @>>> s_*(G_n[L])  @>>>  \cQ'\end{CD}.$$ Then claim 3~\ref{claim3} implies that the middle vertical map is a weak equivalence, and the claim~\ref{claim2} implies that $G$ is a weak equivalence. 

Now we assume that $J$ is an acyclic cofibration from set $A2$ and consider the following pushout square: $$\begin{CD}
\Xi(\cI)  @>>F>       \cP \\
@VVJV               @VVPV\\
\Xi(\cH)   @>>G>       \Q
\end{CD}$$. \\
Recall that $\Xi(I)$ denotes the $\C$-enriched category $I$, viewed as a multicategory via the functor $\Xi(-)$. The category $I$ has one object, $\{0\}$, and endomorphism monoid $I(0,0)=\unit$. The $\C$-enriched category $H$ is has two objects $\{0,1\}$ whose only non-trivial operation is $\cH(0,1)=\unit$. At the level of objects, our map $J$ is the inclusion map $\{0\}\hookrightarrow\{0,1\}$. We pullback along $J$ to factor the diagram into two pushouts:
$$\begin{CD}
\Xi(\cI)  @>>F>       \cP \\
@VVJ'V               @VVP'V\\
J^{*}(\Xi\cH)  @>>>      \R\\
@VVJ''V              @VVP''V\\
\Xi(\cH) @>>G>       \Q
\end{CD}$$. \\

In the top pushout, we claim that $J'$ is a weak equivalence in the model structure on $\multi(\C)_{\{0\}}$ (i.e. the model structure on simplicial operads). It is clear that $F$ is bijective on objects, and so $P'$ is a weak equivalence in $\multi(\C)_{S}$ (where, as usual, $S:=\obj(\cP)$).

Now, the map $J''$ is fully faithful in each simplicial degree and bijective on objects, so by applying claim $4$ degreewise, we conclude that $P''$ is fully faithful. By composition, $P:\cP\longrightarrow\cQ$ satisfies condition $W1$. Since $\pi_0$ preserves pushouts and $\pi_0(J)$ is a trivial cofibration, it follows from the model structure on $\multi$, that condition $W2$ is satisfied. \end{proof}We have now proved the lemma.\end{proof}

We can now prove the existence of the model category structure on the category of all small simplicial multicategories.\\

\begin{proof}[Proof of Theorem]\label{proofofmaintheorem}The category $\multi(\C)$ is co-complete (see,~\cite{EM}), and one can quickly check that all three classes of multifunctors are closed under retracts and that the class of weak equivalences satisfies the ``2-out-3'' property. \\

Given an arbitrary multifunctor we can apply the small object argument to produce a factorization $F=P\circ I$ where $I$ is in $\bar{I}$ and $P$ has the right lifting property with respect to $I$. The our lemma~\ref{classifying trivial fibrations} implies that $P$ is an acyclic fibration. In a similar manner, we factor $F=Q\circ J$, where $J$ is in $\bar{J}$ and $Q$ has the right lifting with respect to $Q$. The lemma~\ref{classifyingfibrations} implies that $Q$ is a fibration of multicategories.\\

Finally, we check that given the square 

$$\begin{CD}
\cA @>F>> \cP\\
@VVIV        @VVPV\\
\cB @>G>> \Q
\end{CD}$$ with $I$ a cofibration and $P$ a fibration of multicategories. If $P$ is also a weak equivalence, then we find a lift by the classification of cofibrations. If $I$ is a weak equivalence, then we factor $I=Q\circ J:\cA\hookrightarrow\widetilde{\cB}\twoheadrightarrow\cB$ where $Q$ is a fibration and $J$ is an acyclic cofibration. Since we have shown that every acyclic cofibration is a weak equivalence, we know that $J$ is a weak equivalence. The ``2-out-of-3'' property for weak equivalences now implies that $Q$ is an acylic fibration.\\

Since $J$ is an acyclic cofibration, and $P$ is a fibration, we know that $P$ has the RLP with respect to $J$.  In other words, we have a lift $H:\widetilde{\cB}\longrightarrow\cP$ so that $J\circ H =F$ and $Q\circ G=P\circ H$. \\

Now, since $I$ is an acyclic cofibration and $Q$ is a trivial fibration, there exists a retract $S$ of $Q$ with $S\circ I = J$. The composite $H\circ S$ provides the desired lift.\end{proof}

\end{document}